\documentclass[11pt]{amsart}

\usepackage[pdftex]{color,graphicx}
\usepackage[utf8]{inputenc}
\usepackage[english]{babel}
\usepackage{enumerate}
\usepackage{cleveref}

\usepackage{amsmath}
\usepackage{amsfonts}
\usepackage{amssymb}
\usepackage{amsthm}
\usepackage{mathrsfs}
\usepackage{mathtext}
\usepackage{mathtools}
\usepackage[mathscr]{eucal}

\newcommand{\cJ}{\mathcal{J}}
\newcommand{\cT}{\mathcal{T}}
\newcommand{\Spc}{\mathrm{Spc}}

\newcommand{\Sp}{\mathrm{Sp}}

\theoremstyle{plain}
\newtheorem{theo}[equation]{Theorem}

\theoremstyle{definition}

\newtheorem{defi}[equation]{Definition}

\theoremstyle{remark}
\newtheorem{rema}[equation]{Remark}
\newtheorem*{ack}{Acknowledgments}

\title[Universal support for triangulated categories]{Universal support for triangulated categories}

\author[P.\ Balmer]{Paul Balmer}
\address{UCLA Mathematics Department, Los Angeles, CA 90095-1555, USA}
\email{balmer@math.ucla.edu}

\author[P.\ S.\ Ocal]{Pablo S.\ Ocal}
\address{UCLA Mathematics Department, Los Angeles, CA 90095-1555, USA}
\email{socal@math.ucla.edu}

\date{2023 August 21}

\subjclass[2020]{18F99; 18G80}

\keywords{Triangulated category, thick subcategory, spectrum, support.}

\thanks{The first author was supported by NSF grant DMS-215375. The second author was supported by an AMS-Simons Travel Grant.}

\begin{document}

\begin{abstract}
We revisit a result of Gratz and Stevenson on the universal space that carries supports for objects of a triangulated category, in the absence of a tensor product.
\end{abstract}

\maketitle

Throughout, $\cT$ denotes an essentially small triangulated category.
In~\cite{balmer:spectrum}, it was established that if $\cT$ is moreover a \emph{tensor}-triangulated category then there exists a universal (final) space~$X$ carrying a support datum for objects of~$\cT$. In that context, a support datum on~$\cT$ consists of closed subsets~$\sigma(a)\subseteq X$ for every object~$a$ in~$\cT$ satisfying
\begin{enumerate}[(SD\,1)]
\item
\label{it:SD1}%
$\sigma(0) = \emptyset$,
\smallbreak
\item
\label{it:SD2}%
$\sigma(a\oplus b) = \sigma(a)\cup \sigma(b)$ for all objects $a,b\in \cT$,
\smallbreak
\item
\label{it:SD3}%
$\sigma(\Sigma a) = \sigma(a)$ for all objects $a\in \cT$,
\smallbreak
\item
\label{it:SD4}%
$\sigma(a)\subseteq \sigma(b)\cup\sigma(c)$ for every distinguished triangle $a\to b\to c\to \Sigma a$ of~$\cT$,
\smallbreak
\item
\label{it:SD5}%
$\sigma(1_{\otimes})=X$,
\smallbreak
\item
\label{it:SD6}%
$\sigma(a\otimes b)=\sigma(a)\cap \sigma(b)$ for all objects $a,b\in\cT$.
\end{enumerate}
This universal space, denoted~$\Spc(\cT)$ and called the \emph{tt-spectrum} of~$\cT$, has become the basis of tensor-triangular geometry.
An early theorem in that field is that $\Spc(\cT)$ can be used to describe the lattice of thick $\otimes$-ideal subcategories of~$\cT$. (We say `thick' to abbreviate `triangulated and thick'.)

In view of this result, Buan-Krause-Solberg~\cite{buan-krause-solberg:ideal} and Kock-Pitsch~\cite{kock-pitsch:stone} argued that lattice theory should play the primordial role in this story.
However Br{\"u}ning~\cite{bruning:hereditary} showed early on that, without tensor, the lattice of thick subcategories is simply not distributive in general, already for the derived category of the~$A_2$ quiver.
Stevenson also pointed out that the same issue occurs in algebraic geometry, e.g.\ for the projective line.

Leaving lattices aside, we can return to the original motivation and ask for a universal space carrying supports for objects of~$\cT$ without the tensor-related axioms~(SD\,\ref{it:SD5}) and~(SD\,\ref{it:SD6}).
And if such a universal space exists, we can ask whether it recovers the lattice of thick subcategories.
After answering these questions, we realized that Gratz and Stevenson~\cite{gratz-stevenson:approximating} implicitly provide a solution, namely by combining their Proposition~5.3.3 and their Subsection~5.4.
We hope there is some didactic value in spelling out our answer.

\begin{defi}
\label{defi:sup-datum}%
A \emph{support datum} on~$\cT$ is a pair $(X,\sigma)$ where $X$ is a topological space and $\sigma : \mathrm{Obj}(\cT)\to \mathrm{Closed}(X)$ associates to every object $a\in \cT$ a closed subset $\sigma(a) \subseteq X$ satisfying Conditions~(SD\,\ref{it:SD1})--(SD\,\ref{it:SD4}).
A \emph{morphism of support data} $f:(X,\sigma)\to (Y,\tau)$ is a continuous map $f:X\to Y$ such that $\sigma(a) = f^{-1}(\tau(a))$ for all objects $a \in \cT$.
\end{defi}

The good news is: The universal support theory always exists.

\begin{theo}
\label{Thm:Sp}%
Every essentially small triangulated category~$\cT$ admits a final support datum.
\end{theo}

\begin{proof}
Now comes the bad news. Define the set~$\Sp(\cT)$ as follows
\begin{equation}
\label{eq:SpT}%
\Sp(\cT) = \{\cJ\subseteq \cT \,|\, \cJ\text{ thick subcategory of }\cT \}
\end{equation}
together with the subsets~$\sup(a)= \{\cJ\in \Sp(\cT) \,|\, a\notin \cJ \}$ for every object~$a\in \cT$.
These give a final support datum on~$\cT$ if we define a topology on~$\Sp(\cT)$ as having~$\{\sup(a)\}_{a\in \cT}$ as basis of closed subsets.
Let $(X,\sigma)$ be a support datum on $\cT$. Define~$f\colon X\to \Sp(\cT)$ by
\begin{equation*}
f(x)=\{a\in\cT \,|\, x\notin \sigma(a)\}
\end{equation*}
for every~$x\in X$. From $\sigma$ being a support datum one checks that $f(x)$ is a thick subcategory, so $f$ is well-defined.
By definition of~$\sup(-)$ and of~$f$ we have for every~$a\in\cT$
\begin{equation*}
f^{-1}(\sup(a)) = \{x\in X \,|\, f(x)\in\sup(a) \} = \{x\in X \,|\, a\notin f(x) \} = \{x\in X \,|\, x\in\sigma(a) \} = \sigma(a).
\end{equation*}
This implies that $f:X\to\Sp(\cT)$ is continuous, and a morphism of support data.
For uniqueness, note that if $g:(X,\sigma)\to(\Sp(\cT),\sup)$ is another such morphism then we have
\begin{align*}
g(x) &= \{a\in \cT \,|\, a\in g(x) \} = \{a\in \cT \,|\, g(x)\notin \sup(a) \}\\
 &= \{a\in \cT \,|\, x\notin g^{-1}(\sup(a)) \} = \{a\in \cT \,|\, x\notin \sigma(a) \} = f(x)
\end{align*}
for all $x\in X$, using the definitions of~$\sup$ and of~$f$, together with~$g^{-1}(\sup(a))=\sigma(a)$.
\end{proof}

\begin{rema}
The space $\Sp(\cT)$ does `classify' all thick subcategories of~$\cT$ but in an extremely dull way: It \emph{is} itself the whole lattice.
Unfortunately, the compression of information that prime ideals allow in the tensor setting does not happen anymore.
And this is not because of whimsical choices in the construction~\eqref{eq:SpT}. Being the solution to a universal problem, the space~$\Sp(\cT)$ does not involve any choice. It is what it is.
\end{rema}

\begin{ack}
We thank Greg Stevenson for explanations about~\cite{gratz-stevenson:approximating}. We thank Henning Krause for mentioning our viewpoint in~\cite{krause:analogue}.
\end{ack}


\end{document}